\documentclass{amsart}
\usepackage{amssymb}
\usepackage{amsmath,amsfonts,amsthm}
\usepackage{tikz}
\usepackage{graphicx}
\usepackage{capt-of}
\usepackage{lipsum}
\graphicspath{ {./images/} }

\newtheorem{theorem}{Theorem}[section]
\newtheorem{lemma}[theorem]{Lemma}
\newtheorem{corollary}[theorem]{Corollary}
\newtheorem{proposition}[theorem]{Proposition}

\theoremstyle{definition}

\newtheorem{example}[theorem]{Example}

\theoremstyle{remark}
\newtheorem{remark}[theorem]{Remark}

\numberwithin{equation}{section}

\newcommand*\xbar[1]{%
	\hbox{%
		\vbox{%
			\hrule height 0.5pt % The actual bar
			\kern0.5ex%         % Distance between bar and symbol
			\hbox{%
				\kern-0.1em%      % Shortening on the left side
				\ensuremath{#1}%
				\kern-0.1em%      % Shortening on the right side
			}%
		}%
	}%
}

\begin{document}
	\title{Rigidity of Bach-Flat Gradient Schouten Solitons}
	
	\author{Valter Borges}
	\address{Department of Mathematics, 70910-900, Bel\'em-PA, Brazil}
	%\curraddr{Department of Mathematics and Statistics,
	%Case Western Reserve University, Cleveland, Ohio 43403}
	\email{valterborges@mat.ufpa.br}
	%\thanks{The author was supported in part by CNPq Proc. 248877/2013-5  Ministry of Science and Technology, Brazil.}
	
	%    General info
	
	\date{\today}
	
\begin{abstract}
In this paper we show that a complete Schouten soliton whose Ricci tensor has at most two eigenvalues at each point is rigid. This allows the classification of both shrinking and expanding Bach-flat Schouten solitons for $n\geq4$. When $n=3$ we are able to conclude rigidity under a more general condition, namely when the Bach tensor is divergence free. These results imply rigidity of locally conformally flat Schouten solitons for $n\geq3$.

%This result follows from a more general one in which we prove that whenever the Ricci tensor of such a manifold has constant rank equals to $n-1$ or $n$, then the scalar curvature is constant, from where we conclude rigidity using previous results concerning Ricci solitons.
\end{abstract}
\maketitle
\section{Introduction and main results}

A Riemannian manifold $(M^n,g)$ endowed with a smooth function $f\in C^{\infty}(M)$ and for which there is a constant $\lambda$ satisfying the tensorial equation
\begin{align}\label{fundeq}
Ric+\nabla\nabla f=\left(\frac{R}{2(n-1)}+\lambda\right)g
\end{align}
is called a {\it gradient Schouten soliton}, where $f$ is called its {\it potential function}. In the equation above $\nabla\nabla f$ is the Hessian of $f$, $R$ is the scalar curvature and $Ric$ is the Ricci tensor of $M$. The soliton is called {\it shrinking}, {\it steady} or {\it expanding}, provided $\lambda$ is positive, zero or negative, respectively. In this case we use the notation $(M^n,g,f,\lambda)$.

Schouten solitons were introduced in \cite{catino} by Catino and Mazzieri. In the same paper they introduced the {\it $\rho$-Einstein solitons}, which are depicted as the Riemannian manifolds for which given $\rho\in\mathbb{R}$, there are $f\in C^{\infty}(M)$, also called {\it potential function}, and $\lambda\in\mathbb{R}$, satisfying
\begin{align}\label{genfundeq}
	Ric+\nabla\nabla f=\left(\rho R+\lambda\right)g.
\end{align}
For $\rho=1/2(n-1)$ one recovers $(\ref{fundeq})$, while for $\rho=0$ one obtains the famous Ricci solitons \cite{hamilton1}, which have been intensively studied in the recent years. It is important to mention that different values of $\rho$ may give rise to quite different objects. For instance, if $\rho\neq0$ then the corresponding $\rho$-Einstein manifold is rectifiable (see \cite{catino} for the definition) and when $\rho\in\{1/2,1/2(n-1),1/n\}$ the corresponding ones which are compact must have constant potential function. Another difference is that, while it is known that for $\rho\notin\{1/2(n-1),1/n\}$ the corresponding solitons are analytic, nothing is known in the remaining cases. For the proofs see \cite{catino,catino1}.

$\rho$-Einstein solitons are closely related to certain geometric evolution equations, for they are the static formulations of self-similar solutions of such equations \cite{catino1}, whose definition has dynamic nature. It is important to mention that among these evolution equations, the one for which $(\ref{fundeq})$ represents a self-similar solution is still an enigma concerning basic properties such as short time existence \cite{catino2}.

%Schouten solitons also appeared in the study of {\it almost Schur lemma} \cite{peterlelis}, an $L^2$ inequality true for compact manifolds of nonnegative Ricci curvature, whose equality is achieved only on Einstein manifolds. In order to describe the equality, a first step was to prove that the metric was a Schouten solitons. 

Besides rectifiability \cite{catino}, other general results on Schouten solitons were obtained by the author in the manuscript \cite{borges}. Among other results, the author has shown in \cite{borges} that the scalar curvature of such metrics are bounded and have a defined sign. It has also been proved that the potential function controls the growth of the squared norm of its gradient. See Theorem \ref{maintheorem} in Section \ref{preliminariesresults} below.

Concerning examples of Schouten solitons, the simplest ones are Einstein manifolds. Other examples are obtained as follows.
\begin{example}\label{example}
	Given $n\geq3$, $k\leq n$ and $\lambda\in\mathbb{R}$, consider an Einstein manifold $(N^{k},g)$ of dimension $k\leq n$ and scalar curvature
	\begin{align}\label{scacurv}
		R_{N}=\frac{2(n-1)k\lambda}{2(n-1)-k}.
	\end{align}
	Now, if $(x,p)\in\mathbb{R}^{n-k}\times N^{k}$, $\|x\|^2$ denotes the Euclidean norm, and
	\begin{align}\label{potfuc}
	f(x,p)=\frac{1}{2}\left(\frac{R_{N}}{2(n-1)}+\lambda\right)\|x\|^2	,
	\end{align}
	it follows that $(\mathbb{R}^{n-k}\times_{\Gamma} N^k,g,f,\lambda)$ is an $n$ dimensional Schouten soliton, where $g=\left\langle ,\right\rangle+g_{N}$, and $\Gamma$ acts freely on $N$ and by orthogonal transformations on $\mathbb{R}^{n-k}$. If $k=0$, $(\mathbb{R}^{n},g,f,\lambda)$ will be addressed as the {\it Gaussian soliton}.
\end{example}

A Riemannian manifold is called {\it rigid} if it is isometric to one of those described in Example \ref{example}. The following classes of Schouten solitons were proven to be rigid in \cite{catino}: {\it compact}; {\it complete noncompact with $\lambda=0$ and $n\geq3$}; {\it complete noncompact with $\lambda>0$ and $n=3$}; {\it complete warped products $B^1\times_{h}N^{n-1}$ (including those with rotational symmetry), where $B^1$ is a one dimensional manifold, $h:B^1\rightarrow\mathbb{R}$ is a positive smooth function and $N^{n-1}$ a space form.}

In this paper we will add to the list above Schouten solitons whose Ricci tensors have at most two principal curvatures at each point. Namely,
\begin{theorem}\label{criteria}
	Let $(M^n,g,f,\lambda)$, $\lambda\neq0$, be a complete Schouten soliton with $n\geq3$ and assume that its Ricci tensor has at most two eigenvalues at each point of $M$. Then one of the following is true
	\begin{enumerate}
		\item $M$ is Einstein with scalar curvature given by $(\ref{scacurv})$ for $k=n$, and $f$ is constant;
		\item $M$ is isometric to a finite quotient of the Gaussian Soliton, with potential function given by $(\ref{potfuc})$;
		\item\label{3item} $M$ is isometric to a finite quotient of $\mathbb{R}\times N^{n-1}$, where $N^{n-1}$ is Einstein if $n\geq4$, and a space form if $n=3$. Furthermore, the scalar curvature and the potential function are given by $(\ref{scacurv})$ and $(\ref{potfuc})$ for $k=n-1$.
	\end{enumerate}
	In particular, $(M^n,g,f,\lambda)$ is rigid.
\end{theorem}

In order to prove the theorem above, we will show that the scalar curvature of such manifolds are constant, which implies that their Ricci tensors have constant rank equals to $0$, $n-1$ or $n$. Previous results on rigidity of Ricci solitons are used to conclude the proof \cite{ferlo}.

In what follows we will rewrite the condition of Theorem \ref{criteria} on the eigenvalues of the Ricci tensor. As it was proved by Catino and Mazzieri \cite{catino}, wherever $\nabla f$ does not vanish, it is an eigenvector of $Ric$, whose corresponding eigenvalue is $0$. Therefore, the hypothesis of having at most two eigenvalues turns out to be equivalent to
\begin{align}\label{atmost2}
	R^2-(n-1)|Ric|^2=0
\end{align}
on the set of regular points of $M$. In \cite{borges} the author has shown that the set of regular points of $f$ is dense in $M$ (see also Proposition \ref{dense} below), which implies that $(\ref{atmost2})$ is equivalent to the Ricci curvature of $M$ having at most two eigenvalues at each point of $M$.
% We note that the density of the set of regular points of $f$ is an important tool in investigating Schouten solitons, since we do not know whether these manifolds are analytical, a property that turns out to be useful when setting rigidity.
%\begin{enumerate}
%	\item {\it compact};% Case where it is Einstein;
%	\item {\it complete noncompact with $\lambda=0$ and $n\geq3$};%. In this case it is Ricci flat;
%	\item {\it complete noncompact with $\lambda>0$ and $n=3$};
%	\item {\it complete warped products $B^1\times_{h}N^{n-1}$ (including those with rotational symmetry), where $B^1$ is a one dimensional manifold, $h:B^1\rightarrow\mathbb{R}$ is a positive smooth function and $N^{n-1}$ a space form.}
%\end{enumerate}

Now we turn to the classification of Bach-flat Schouten solitons. Recall that steady Schouten solitons were already proven to be Ricci flat. Therefore, we concentrate our study in classifying shrinking and expanding Schouten solitons. There is a wildly known classification of Bach-flat Ricci solitons. In \cite{caoqian}, Cao and Chen showed rigidity of Bach-flat shrinking Ricci solitons. In \cite{caocatino}, Cao, Catino and coworkers classified steady Ricci solitons, showing that they must be either rigid or homothetic to the Bryant soliton. It was also proved in \cite{caocatino} that Bach-flat expanding Ricci solitons with nonnegative Ricci curvature are rotationally symmetric. Although, there are non rigid examples in the latter case (see page 12 of \cite{caocatino}).

The rigidity of Bach-flat Schouten solitons is a consequence of Theorem \ref{criteria}. More precisely, we will show that $\nabla f$ is an eigenvalue of the Bach tensor of a Schouten soliton at regular points of $f$, whose corresponding eigenvalue is, up to a constant, given by the expression in left hand side of $(\ref{atmost2})$. This connects Theorem \ref{criteria} and Theorem \ref{vanisbach}, and shows that the same result is true if the Bach tensor vanishes in the direction of $\nabla f$.
%\begin{theorem}\label{vanisbach}
%	Let $(M^n,g,f,\lambda)$, $n\geq3$, be a complete noncompact Schouten soliton with $\lambda\neq0$. If its Bach tensor vanishes, then it is isometric to one of those listed in Theorem \ref{criteria}. Besides, if $n=3$, then $N^2$ in item $(\ref{3item})$ must have constant sectional curvature.
%\end{theorem}
\begin{theorem}\label{vanisbach}
	 A complete noncompact Bach-flat Schouten soliton $(M^n,g,f,\lambda)$ with $n\geq4$ and $\lambda\neq0$ is isometric to one of those listed in Theorem \ref{criteria}.
\end{theorem}

Notice that unlike expanding Bach-flat Ricci solitons, an expanding Schouten soliton is rigid under the vanishment of its Bach tensor.

For $n=3$ we follow the same definition of Bach tensor introduced in \cite{caocatino} (see $(\ref{bhctn=3})$). With this definition, Theorem \ref{vanisbach} is also true for a $3$ dimensional Schouten soliton, and the proof does not change. But a stronger result is true, as we state below.

%\begin{theorem}\label{vanisdivbach}
%	Let $(M^3,g,f,\lambda)$ be a complete noncompact Schouten soliton with $\lambda\neq0$. If its Bach tensor is divergence-free, then $(\ref{atmost2})$ holds on $M^3$. Then it must be isometric to one of those listed in Theorem \ref{criteria}. Besides, $N^2$ in item $(\ref{3item})$ must have constant sectional curvature.
%\end{theorem}
\begin{theorem}\label{vanisdivbach}
	A complete noncompact Schouten soliton $(M^3,g,f,\lambda)$ with $\lambda\neq0$, whose Bach tensor is divergence-free, is isometric to one of those listed in Theorem \ref{criteria}. Besides, $N^2$ in item $(\ref{3item})$ is a space form.
\end{theorem}

%\begin{theorem}\label{vanisdivbach}
%	Let $(M^3,g,f,\lambda)$ be a complete noncompact Schouten soliton with $\lambda\neq0$. If its Bach tensor is divergence-free, then $(M^3,g)$ is either Einstein and $f$ is constant, or $(M^n,g)$ is isometric to the Gaussian soliton or it is isometric to $\mathbb{R}\times N^{2}$, where $N^2$ has constant sectional curvature.
%\end{theorem}

\begin{remark}
  Shrinking Schouten solitons have already been classified in dimension $3$ by Catino et all (see Theorem 5.4 in \cite{catino}). Therefore, when $n=3$ the novelty of the theorems above concerns only expanding Schouten solitons.
\end{remark}

It is well known that locally conformally flat manifolds are necessarily Bach-flat. Therefore, a direct consequence of the previous results is the classification of locally conformally flat shrinking and expanding Schouten solitons.

%\begin{corollary}
%	Let $(M^n,g,f,\lambda)$, $\lambda\neq0$, be a complete noncompact Schouten soliton whose Weyl tensor vanishes and $n\geq3$. Then it must be isometric to one of those listed in Theorem \ref{criteria}, where $N^{n-1}$ in item $(\ref{3item})$ is a space form.
%\end{corollary}
\begin{corollary}\label{confclass}
	A locally conformally flat complete noncompact Schouten soliton $(M^n,g,f,\lambda)$ with $n\geq3$ and $\lambda\neq0$ is isometric to one of those listed in Theorem \ref{criteria}, where $N^{n-1}$ in item $(\ref{3item})$ is a space form.
\end{corollary}

This paper is organized in the following way. In Section \ref{preliminariesresults} we introduce notation, definition and provide the basic tools to prove the main results of this paper. In Section \ref{sectionequality} we deal with Schouten solitons whose Ricci tensors have at most two eigenvalues at each point. We close this section with the proof of Theorem \ref{criteria}. In Section \ref{sectionbachT} we recall the definition of classical tensors as the Weyl, Cotton and Bach tensor (including its extension to $n=3$) present in Theorem \ref{vanisbach}, Theorem \ref{vanisdivbach} and Corollary \ref{confclass} and present their proofs. It is in this section that we show that $\nabla f$ is an eigenvalue of the Bach tensor of $M$ at regular points of $f$.

\section{Preliminary Results}\label{preliminariesresults}
The proposition below collects some important identities on Schouten solitons.

\begin{proposition}[\cite{catino}]If $(M^n,g,f,\lambda)$ is a gradient Schouten soliton, then
	%\begin{align}\label{wedgezero}
	%	df\wedge dR=0,
	%\end{align}
	\begin{align}\label{trace}
		\Delta f=n\lambda-\frac{n-2}{2(n-1)}R,
	\end{align}
	\begin{align}\label{Riccizero}
		Ric(\nabla f,X)=0,\ \forall X\in\mathfrak{X}(M),
	\end{align}
	\begin{align}\label{IdentitySch}
		\left\langle\nabla f,\nabla R\right\rangle+\left(\frac{R}{n-1}+2\lambda\right)R=2|Ric|^2.
	\end{align}
\end{proposition}

In \cite{borges} the author proved the following result, which can be seen as the analog of Hamilton's identity for Ricci solitons. The latter plays a fundamental role to Ricci soliton's theory, as one can see for example in \cite{pointofview,petersen} and references therein.
\begin{theorem}[\cite{borges}]\label{maintheorem}
	Let $(M^n,g,f,\lambda)$, $\lambda\neq0$, be a complete noncompact Schouten soliton with $f$ nonconstant. If $\lambda>0$ $($$\lambda<0$, respectively$)$, then the potential function $f$ attains a global minimum $($maximum, respectively$)$ and is unbounded above $($below, respectively$)$. Furthermore,
	\begin{align}\label{positivity}
		0\leq \lambda R\leq2(n-1)\lambda^2,
	\end{align}
	\begin{align}\label{gradestimate}
		2\lambda(f-f_{0})\leq |\nabla f|^2\leq4\lambda(f-f_{0}),
	\end{align}
	with $f_{0}=\displaystyle\min_{p\in M}f(p)$, if $\lambda>0$, $($$f_{0}=\displaystyle\max_{p\in M}f(p)$, if $\lambda<0$, respectively$)$.
\end{theorem}
For applications of the result above see \cite{borges}, where volume growth of geodesic balls are investigated for shrinking Schouten solitons.

In order to prove Theorem \ref{maintheorem}, an ordinary differential inequality satisfied by $|\nabla f|^2$ along suitable curves was important. Let us recall such inequality, since it will be used in next sections.

Let $p\in M$ be a regular point of $f$ and $\alpha:(\omega_{1},\omega_{2})\rightarrow M$ the maximal integral curve of $\frac{\nabla f}{|\nabla f|^2}$ through $p$. It is not hard to see that
	\begin{align}\label{lindepen}
	(f\circ\alpha)'(s)=1,\ \forall s\in(\omega_{1},\omega_{2}),
	\end{align}
that is, $f\circ\alpha$ is a linear function of $s$.

\begin{proposition}[\cite{borges}]\label{strategylemma}
	Let $(M^n,g,f,\lambda)$, $\lambda\neq0$, be a Schouten soliton with $f$ nonconstant and $\alpha(s)$, $s\in(\omega_{1},\omega_{2})$, a maximal integral curve of $\frac{\nabla f}{|\nabla f|^2}$. The function $b:(\omega_{1},\omega_{2})\rightarrow\mathbb{R}$, defined by
	\begin{align}\label{composition}
	b(s)=|\nabla f(\alpha(s))|^2.
	\end{align}
	satisfies the differential inequality
	\begin{equation}\label{MainIneq}
	bb''-(b')^2+6\lambda b'-8\lambda^2\geq0,
	\end{equation}
	where $b'$ and $b''$ are the first and the second derivative of $b$ with respect to $s$. Furthermore, equality holds in $(\ref{MainIneq})$ at $s_{1}$ if and only if $(n-1)|Ric|^2=R^2$ holds on $f^{-1}(s_{1})$.
\end{proposition}
\begin{proof}
	Consider the smooth function $a:(\omega_{1},\omega_{2})\rightarrow\mathbb{R}$ given by $a(s)=R(\alpha(s))$. From $d(|\nabla f|^2)(X)=2\nabla\nabla f(X,\nabla f)$ and equation $(\ref{fundeq})$ one has
	\begin{align}\label{id}
		\begin{split}
			b'(s)=d(|\nabla f|^2)(\alpha'(s))=\frac{a(s)}{n-1}+2\lambda,
		\end{split}
	\end{align}
	which after differentiating gives $a'(s)=(n-1)b''(s)$. Consequently,
	\begin{align}\label{int3}
		\begin{split}
			\left\langle\nabla f(\alpha(s)),\nabla R(\alpha(s))\right\rangle=b(s)a'(s)=(n-1)b(s)b''(s).
		\end{split}
	\end{align}
	Putting $(\ref{IdentitySch})$, $(\ref{id})$ and $(\ref{int3})$ together we have
	\begin{align}\label{studyequality}
		(n-1)(b(s)b''(s)+b'(s)(b'(s)-2\lambda))&=2|Ric|^2(\alpha(s))\\
		&\geq2(n-1)(b'(s)-2\lambda)^2,\nonumber
	\end{align}
	where in the second line we have used the inequality $(n-1)|Ric|^2\geq R^2$, which is a consequence of $(\ref{Riccizero})$, and $(\ref{id})$ once again. Consequently,
	\begin{equation}\label{theend}
		bb''\geq(b'-2\lambda)(b'-4\lambda).
	\end{equation}
	Suppose there is an $s_{1}$ where equality is reached in $(\ref{MainIneq})$. Then equality is also obtained in $(\ref{studyequality})$, what is only possible if $(n-1)|Ric|^2=R^2$ on $f^{-1}(s_{1})$.
\end{proof}

%\begin{remark}\label{simply}
%	If $b(s)$, $s\in(s_{1},s_{2})$, is a solution of $(\ref{MainIneq})$ with $\lambda<0$, then $\phi(z)=b(-z)$, $z\in(-s_{2},-s_{1})$, satisfies
%	\begin{align*}
%	\phi\phi_{zz}-\left(\phi_{z}\right)^2+6\mu\phi_{z}-8\mu^2\geq0,
%	\end{align*}
%	with $\mu=-\lambda>0$. Here $\phi_{z}$ is the derivative of $\phi$ with respect to $r$.
%\end{remark}
%
\begin{example}\label{examplater}
Let $(\mathbb{R}^{n-k}\times_{\Gamma} N^k,g,f,\lambda)$ be the Schouten soliton of Example \ref{example}. Notice that it has constant scalar curvature $R=\frac{2(n-1)k\lambda}{2(n-1)-k}$ and, if $k\leq n-1$, its potential function $f$ is not constant and satisfies $|\nabla f|^2=\left(\frac{R}{n-1}+2\lambda\right)f$. After a linear change of coordinates using $(\ref{lindepen})$ with the condition $f(\alpha(0))=0$ and replacing $f$ by $s$, we obtain the function $b(s)=\left(\frac{R}{n-1}+2\lambda\right)s$. Now, a simple computation shows that $b'(s)=\frac{4(n-1)\lambda}{2(n-1)-k}$, and then
\begin{align}\label{goalgoal}
	(b'-2\lambda)(b'-4\lambda)=-\frac{8k(n-1-k)\lambda^2}{(2(n-1)-k)^2}\leq0.
\end{align}
Once $bb''=0$, for $b'$ is constant, we see that $b(s)$ is a solution of $(\ref{MainIneq})$. Observe that equality holds in $(\ref{goalgoal})$ for $k=0$ and $k=n-1$.
\end{example}

We take the opportunity to mention that the inequality in $(\ref{goalgoal})$, that is, $$(b'-2\lambda)(b'-4\lambda)\leq0,$$ was proven to be true for all complete Schouten solitons whose potential function is not constant \cite{borges}. This is an important tool used in the proof of Theorem \ref{maintheorem}.

\section{Equality for $(\ref{MainIneq})$ and proof of Theorem \ref{criteria}}\label{sectionequality}

In this section we investigate the geometry of $M$ when equality holds in $(\ref{MainIneq})$ for any $s\in(\omega_{1},\omega_{2})$. More precisely, we will assume that $b(s)=|\nabla f(\alpha(s))|^2$ is a solution of 
\begin{equation}\label{Maineq}
	bb''-(b')^2+6\lambda b'-8\lambda^2=0,
\end{equation}
for each $s\in(\omega_{1},\omega_{2})$. As we know from Proposition \ref{strategylemma}, this is equivalent to $(n-1)|Ric|^2=R^2$ on the set of regular points of $f$. If we could approximate any critical point by a sequence of regular points, this equivalence would be true all over $M$. In the next proposition we show that this is actually the case.

\begin{proposition}\label{dense}
	Let $(M^n,g,f,\lambda)$, $\lambda\neq0$, be a complete Schouten soliton with $f$ nonconstant. The set of regular points of $f$ is dense in $M$.
\end{proposition}
\begin{proof}
	Denote by $\mathcal{R}$ the set of regular points of $f$. Suppose by contradiction that the open set $U=M\backslash\overline{\mathcal{R}}$ is nonempty. Then $\Delta f$ vanishes in $U$, what from  $(\ref{trace})$ and $(\ref{positivity})$ gives, at $p\in U$, the following inequality
	\begin{align*}
		\frac{2(n-1)n\lambda^2}{n-2}=R(q)\lambda\leq2(n-1)\lambda^2,
	\end{align*}
	which cannot be true. Then $\overline{\mathcal{R}}=M$, as the proposition claims.
\end{proof}

The lemma below shows an algebraic relation between a solution of $(\ref{Maineq})$ and its first derivative, which can be seen as a first integral of $(\ref{Maineq})$.

\begin{lemma}\label{lemmaintegrability}
	Let $b(s)$ be a smooth solution of $(\ref{Maineq})$ so that $b'(s)\neq2\lambda$ on an interval $I\subset(\omega_{1},\omega_{2})$. Then there is a constant $\sigma_{0}$ so that
	\begin{align}\label{lemmaintegrabilityeq}
		(b'(s)-4\lambda)^2=\sigma_{0}b(s)(b'(s)-2\lambda), \forall s\in I.
	\end{align}
\end{lemma}
\begin{proof}
	A straightforward computation gives
	\begin{align*}
		&\left(\frac{(b'(s)-4\lambda)^2}{b(s)(b'(s)-2\lambda)}\right)'=\\\noalign{\smallskip}
		&\frac{(b'-4\lambda)(2(b'-2\lambda)bb''-(b'-4\lambda)(b'(b'-2\lambda)+bb''))}{b^2(b'-2\lambda)^2}=\\\noalign{\smallskip}
		&\frac{b'(b'-4\lambda)}{b^2(b'-2\lambda)^2}(bb''-(b'-4\lambda)(b'-2\lambda))=0,
	\end{align*}
for all $s\in I$, where we have used $(\ref{Maineq})$. This proves the lemma.
\end{proof}

Using the relations $|\nabla f|^2=b$ and $R=(n-1)(b'-2\lambda)$ we have the following corollary, which rewrites $(\ref{lemmaintegrabilityeq})$.

\begin{corollary}\label{corotobe}
	Let $(M^n,g,f,\lambda)$, $\lambda\neq0$, be a Schouten soliton with $f$ nonconstant for which $(\ref{atmost2})$ happens, and let $A\subset M$ be a set of regular points of $f$ where $R$ does not vanish. Then there is a constant $\sigma_{0}$ so that 
	\begin{align}\label{identmanif}
		(R-2(n-1)\lambda)^2=\sigma_{0}(n-1)R|\nabla f|^2
	\end{align}
	on $A$.
\end{corollary}

We are finally ready to set the main ingredient to prove Theorem \ref{criteria}.

\begin{proposition}\label{mainIng}
	Let $(M^n,g,f,\lambda)$, $\lambda\neq0$, be a Schouten soliton with $f$ nonconstant and $\alpha(s)$, $s\in(\omega_{1},\omega_{2})$, a maximal integral curve of $\frac{\nabla f}{|\nabla f|^2}$. Assume that the function $b:(\omega_{1},\omega_{2})\rightarrow\mathbb{R}$, defined by $b(s)=|\nabla f(\alpha(s))|^2$, satisfies the ODE $(\ref{Maineq})$ on $(\omega_{1},\omega_{2})$. Then $b'\equiv2\lambda$ or $b'\equiv4\lambda$.
\end{proposition}
\begin{proof}
	First notice that $(\ref{Maineq})$ can be rewritten in the following two ways
	\begin{align*}
		(b(b'-2\lambda))'=\frac{2}{b^2}(b(b'-2\lambda))^2\ \ \ \ \text{and}\ \ \ \ (b(b'-4\lambda))'=\frac{2(b'-\lambda)}{b}b(b'-4\lambda).
	\end{align*}
	Since $b$ never vanishes on $(\omega_{1},\omega_{2})$, the ODE's above, satisfied by $b(b'-2\lambda)$ and $b(b'-4\lambda)$, respectively, imply that if $b'(s_{0})\in\{2\lambda,4\lambda\}$ for some $s_{0}\in(\omega_{1},\omega_{2})$, then $b'\equiv2\lambda$ or $b'\equiv4\lambda$ on $(\omega_{1},\omega_{2})$. The same conclusion holds if $b''(s_{0})$ vanishes.
	
	Suppose that $b'$ is not constant. In this case, $b'(s)\notin\{2\lambda,4\lambda\},\ \forall s\in(\omega_{1},\omega_{2})$, which by Proposition \ref{dense} and Corollary \ref{corotobe} implies that $(\ref{identmanif})$ is true on $M$ and $\sigma_{0}\neq0$. In particular, at a critical point $p_{0}$ of $f$ this implies that $R(p_{0})=2(n-1)\lambda$. Computing the Laplacian of both sides of $(\ref{identmanif})$ gives
	\begin{align*}
		&\sigma_{0}(n-1)(R\Delta|\nabla f|^2+|\nabla f|^2\Delta R+2\left\langle\nabla|\nabla f|^2,\nabla R\right\rangle)=\\
		&2(R-2(n-1)\lambda)\Delta R+2|\nabla R|^2,
	\end{align*}
	which at $p_{0}$ reveals that $2\sigma_{0}(n-1)^2\lambda\Delta|\nabla f|^2(p_{0})=2|\nabla R(p_{0})|^2=0$, where we have used that $p_{0}$ is a critical point of $R$, a consequence of Theorem \ref{maintheorem}. Using $(\ref{Riccizero})$ and Bochner's formula we conclude that
	\begin{align*}
		|\nabla\nabla f|^2(p_{0})&=\frac{1}{2}\Delta|\nabla f|^2(p_{0})-\left\langle\nabla\Delta f(p_{0}),\nabla R(p_{0})\right\rangle=0,
	\end{align*}
	from where it follows that
	\begin{align*}
		0=n|\nabla\nabla f|^2(p_{0})\geq(\Delta f(p_{0}))^2=4\lambda^2>0,
	\end{align*}
	which is not possible. This proves that $b'$ must be constant.	
\end{proof}

\begin{proof}[{\bf Proof of Theorem \ref{criteria}}]
	If $f$ is constant, then $M$ is Einstein. Now assume that $f$ is not constant and consider $b(s)$ defined as in $(\ref{composition})$. In view of $(\ref{Riccizero})$, $0$ is always an eigenvalue of $Ric$ at the regular points of $f$. Since the set of regular points of $f$ is dense in $M$ (Proposition \ref{dense}), $Ric$ has at most two eigenvalues at each point if and only if $(\ref{atmost2})$ holds on $M$. This, according to Proposition \ref{strategylemma}, implies that $b$ is a solution of $(\ref{Maineq})$. Applying Proposition \ref{mainIng} and using the relation $R=(n-1)(b'-2\lambda)$ (see $(\ref{id})$), we conclude that $R$ is constant, either equals to $0$ or $2(n-1)\lambda$. Then $(M^n,g,f,\mu)$, $\mu\in\{\lambda,2\lambda\}$, is a Ricci soliton whose Ricci tensor, according to $(\ref{atmost2})$, has constant rank equals either to $0$ or $n-1$. Now we apply Theorem 2 of \cite{ferlo} to conclude that $(M,g)$ is rigid, and the theorem is proved.
\end{proof}

\section{The Bach Tensor of a Schouten Soliton}\label{sectionbachT}
In this section we present the proof of Theorem \ref{vanisbach}, Theorem \ref{vanisdivbach} and Corollary \ref{confclass}.

First let us recall some definitions. Let $(M^n,g)$ be a Riemannian manifold. For any $n\geq3$ its {\it Weyl} and {\it Cotton} tensors are defined, respectively, by
\begin{align}
	W_{ijkl}=&R_{ijkl}-\frac{1}{n-2}(g_{ik}R_{jl}-g_{il}R_{jk}-g_{jk}R_{il}+g_{jl}R_{ik})\label{weyl}\\
	&+\frac{R}{(n-1)(n-2)}(g_{ik}g_{jl}-g_{il}g_{jk}),\nonumber\\
	C_{ijk}=&\nabla_{i}R_{jk}-\nabla_{j}R_{ik}-\frac{1}{2(n-1)}(g_{jk}\nabla_{i}R-g_{ik}\nabla_{j}R)\label{cotton},
\end{align}
and for $n\geq4$ the {\it Bach} tensor is defined by
\begin{align}\label{origbach}
	B_{ij}=&\frac{1}{n-3}\nabla^{k}\nabla^{l}W_{ikjl}+\frac{1}{n-2}R_{kl}W_{i\ j}^{\ k\ l}.
\end{align}
In terms of the Cotton tensor, $B$ can be rewritten by
\begin{align}\label{bhct}
	(n-2)B_{ij}=&\nabla^{k}C_{kij}+R_{kl}W_{i\ j}^{\ k\ l}.
\end{align}
As in \cite{caoqian}, we use the fact that $(\ref{bhct})$ is well defined for $n=3$ to define the Bach tensor in this dimension. More precisely, the Bach tensor of $(M^3,g)$ is defined by
\begin{align}\label{bhctn=3}
	B_{ij}=&\nabla^{k}C_{kij}.
\end{align}
One says that a Riemannian manifold $M^n$ is {\it locally conformally flat} if $n\geq4$ and its Weyl tensor vanishes or, if $n=3$ and its cotton tensor vanishes. $M^n$ is said to be {\it Bach-flat} when its Bach tensor vanishes, $n\geq3$.

The proposition below shows that whenever $p\in M$ is a regular point of $f$, its gradient is an eigenvector of the Bach tensor at this point and gives the corresponding eigenvalue.
\begin{proposition}\label{Propgradeigv}
	Let $(M^n,g,f,\lambda)$ be a gradient Schouten soliton with $n\geq3$. At a regular point of $f$ the Bach tensor of $M$ satisfies
	\begin{align}\label{gradeigv}
		B(\nabla f,X)=\frac{R^2-(n-1)|Ric|^2}{(n-1)(n-2)^2}g(\nabla f,X),
	\end{align}
for all $X\in\mathfrak{X}(M)$. If $n=3$ we have, in addition,
	\begin{align}\label{gradediv}
		(div B)(X)=\frac{R^2-2|Ric|^2}{2}g(\nabla f,X).
	\end{align}
\end{proposition}
\begin{proof}
	In order to prove the proposition we need to find an expression for the Bach tensor of a Schouten soliton. We will first compute the Cotton tensor of such a manifold, and then use definitions $(\ref{bhct})$ and $(\ref{bhctn=3})$ to achieve our goal. According to $(\ref{fundeq})$ and $(\ref{cotton})$ we have
	\begin{align}
		C_{ijk}=&\frac{\nabla_{i}R}{2(n-1)}g_{jk}-\nabla_{i}\nabla_{j}\nabla_{k}f-\frac{\nabla_{j}R}{2(n-1)}g_{ik}+\nabla_{j}\nabla_{i}\nabla_{k}f\nonumber\\
		&-\frac{1}{2(n-1)}(\nabla_{i}Rg_{jk}-\nabla_{j}Rg_{ik})\nonumber\\
		=& \nabla_{j}\nabla_{i}\nabla_{k}f-\nabla_{i}\nabla_{j}\nabla_{k}f\nonumber\\
		=& R_{jikl}\nabla_{l}f.\label{relation}
	\end{align}
	and then, as required in $(\ref{bhct})$ and $(\ref{bhctn=3})$, we have the expression
	\begin{align*}
		\nabla_{i}C_{ijk}=\nabla_{i}R_{klji}\nabla_{l}f+\left(\frac{R}{2(n-1)}+\lambda\right)R_{jk}-R_{jikl}R_{il}.
	\end{align*}
	On the othe hand, since
		\begin{align*}
		\nabla_{i}R_{klji}=&\nabla_{l}R_{kj}-\nabla_{k}R_{lj}\\
						  =&-\nabla_{l}\nabla_{k}\nabla_{j}f+\nabla_{k}\nabla_{l}\nabla_{j}f+\frac{1}{2(n-1)}(\nabla_{l}Rg_{kj}-\nabla_{k}Rg_{lj})\\
						  =&R_{klji}\nabla_{i}f+\frac{1}{2(n-1)}(\nabla_{l}Rg_{kj}-\nabla_{k}Rg_{lj})
	\end{align*}
	we conclude that
	\begin{align*}
		\nabla_{i}R_{klji}\nabla_{l}f&=R_{klji}\nabla_{l}f\nabla_{i}f+\frac{1}{2(n-1)}(\left\langle\nabla R,\nabla f\right\rangle g_{kj}-\nabla_{k}R\nabla_{j}f).
	\end{align*}
	Consequently,
	\begin{align}\label{pt2}
		\nabla_{i}C_{ijk}=&R_{klji}\nabla_{l}f\nabla_{i}f+\frac{1}{2(n-1)}(\left\langle\nabla R,\nabla f\right\rangle g_{kj}-\nabla_{k}R\nabla_{j}f)\\
		&+\left(\frac{R}{2(n-1)}+\lambda\right)R_{jk}-R_{jikl}R_{il}\nonumber
	\end{align}

	To prove $(\ref{gradeigv})$ when $n\geq4$, notice that $(\ref{weyl})$, $(\ref{bhct})$ and $(\ref{pt2})$ together imply
	\begin{align*}
		(n-2)B_{kj}=&R_{klji}\nabla_{l}f\nabla_{i}f+\frac{1}{2(n-1)}(\left\langle\nabla R,\nabla f\right\rangle g_{kj}-\nabla_{k}R\nabla_{j}f)\\
		&+\left(\frac{R}{2(n-1)}+\lambda\right)R_{jk}+(W_{jikl}-R_{jikl})R_{il}\\\nonumber
		=&R_{klji}\nabla_{l}f\nabla_{i}f+\frac{1}{2(n-1)}(\left\langle\nabla R,\nabla f\right\rangle g_{kj}-\nabla_{k}R\nabla_{j}f)\\
		&+\left(\frac{R}{2(n-1)}+\lambda\right)R_{jk}+(W_{jikl}-R_{jikl})R_{il}\\\nonumber
		=&R_{klji}\nabla_{l}f\nabla_{i}f+\frac{1}{2(n-1)}(\left\langle\nabla R,\nabla f\right\rangle g_{kj}-\nabla_{k}R\nabla_{j}f)\\
		&+\left(\frac{R}{2(n-1)}+\lambda\right)R_{jk}-\frac{R}{(n-1)(n-2)}(R_{kj}-Rg_{kj})\\\nonumber
		&+\frac{1}{n-2}(2R_{kl}R_{lj}-RR_{kj}-|Ric|^2g_{kj}).
	\end{align*}
	Using $(\ref{Riccizero})$ we finally have
	\begin{align*}
		(n-2)B_{kj}\nabla_{k}f=\frac{R^2-(n-1)|Ric|^2}{(n-1)(n-2)}\nabla_{j}f,
	\end{align*}
	as it was claimed. Similar computations show the result for $n=3$, where the Bach tensor is defined by $(\ref{bhctn=3})$.
	
%	If $n=3$ we use definition $(\ref{bhctn=3})$ and $(\ref{pt2})$ to obtain
%	\begin{align}
%		B_{kj}=&R_{klji}\nabla_{l}f\nabla_{i}f+\frac{1}{4}(\left\langle\nabla R,\nabla f\right\rangle g_{kj}-\nabla_{k}R\nabla_{j}f)\\
%		&+\left(\frac{R}{4}+\lambda\right)R_{jk}-R_{jikl}R_{il},\nonumber
%	\end{align}
%	and using $(\ref{weyl})$ and $(\ref{Riccizero})$ we conclude that
%	\begin{align*}
%		B_{kj}\nabla_{k}f=&-R_{jikl}R_{il}\nabla_{k}f\\
%		=&-(|Ric|^2g_{jk}-2R_{ij}R_{ik}+RR_{jk}-\frac{R}{2}(Rg_{jk}-R_{jk}))\nabla_{k}f\\
%		=&\frac{1}{2}(R^2-2|Ric|^2)\nabla_{j}f,
%	\end{align*}
%	where we have also used the fact that when $n=3$ the Weyl tensor vanishes. This proves $(\ref{gradeigv})$ for $n=3$.
	
	Now we turn to the divergence of $B$ when $n=3$. We use the following formula
	\begin{align*}
		(div B)_{j}=-R_{il}C_{jil},
	\end{align*}
	true in this dimension, proved in \cite{caocatino}. Using $(\ref{weyl})$ and $(\ref{relation})$ we then get
	\begin{align*}
		(div B)_{j}=&-R_{il}R_{ijlr}\nabla_{r}f=-R_{il}R_{jirl}\nabla_{r}f\\
		=&(|Ric|^2g_{jr}-2R_{ij}R_{ir}+RR_{jr}-\frac{R}{2}(Rg_{jr}-R_{jr}))\nabla_{r}f\\
		=&\left(|Ric|^2-\frac{R^2}{2}\right)\nabla_{j}f,
	\end{align*}
	proving $(\ref{gradediv})$.
\end{proof}

\begin{proof}[{\bf Proofs of Theorem \ref{vanisbach} and Theorem \ref{vanisdivbach}}]
	Let $(M^n,g,f,\lambda)$ be a Schouten Soliton with $n\geq3$, whose Bach tensor vanishes. Since by hypothesis we have $B_{ij}\nabla_{j}f=0$ for $n\geq3$, or $(div B)_{j}=0$ when $n=3$, one has on $M$, according to Proposition \ref{Propgradeigv} and Proposition \ref{dense}, that
	\begin{align*}
		R^2=(n-1)|Ric|^2.
	\end{align*}
	Now we apply item $(\ref{3item})$ Theorem \ref{criteria} to finish the proof.
\end{proof}

\begin{proof}[{\bf Proof of Corollary \ref{confclass}}]
	Let $(M^n,g,f,\lambda)$ be a locally conformally flat Schouten Soliton with $n\geq3$. It follows from $(\ref{origbach})$ and $(\ref{bhctn=3})$ that these manifolds are Bach-flat, from where Theorem \ref{vanisbach} and Theorem \ref{vanisdivbach} imply that these manifolds are isometric to those of Theorem \ref{criteria}. In addition, if item $(\ref{3item})$ happens, locally conformally flatness implies that $N^{n-1}$ must be a space form, and the corollary is proved.
\end{proof}
%
%\begin{align}\label{pt1}
%	R_{jikl}R_{il}=&|Ric|^2g_{jk}-2R_{ij}R_{ik}+RR_{jk}-\frac{R}{2}(Rg_{jk}-R_{jk}).
%\end{align}
%
%\begin{align*}
%	B_{kj}=&R_{klji}\nabla_{l}f\nabla_{i}f+\frac{1}{2}(\left\langle\nabla R,\nabla f\right\rangle g_{kj}-\nabla_{k}R\nabla_{j}f)+\left(\frac{R}{4}+\lambda\right)R_{jk}\\
%	&-|Ric|^2g_{jk}+2R_{ij}R_{ik}-RR_{jk}+\frac{R}{2}(Rg_{jk}-R_{jk})
%\end{align*}

\bibliographystyle{amsplain}

\end{document}